\documentclass[11pt]{article}

\usepackage[margin=1.1in]{geometry}
\usepackage{amsmath,amssymb,amsthm}
\usepackage{hyperref}
\usepackage{cleveref}
\usepackage{algorithm}
\usepackage{algpseudocode}
\usepackage{comment}
\usepackage{graphicx}
\usepackage{listings}
\newtheorem{theorem}{Theorem}
\newtheorem{lemma}{Lemma}

\DeclareMathOperator{\Sp}{Sp}
\DeclareMathOperator{\PSp}{PSp}
\DeclareMathOperator{\GL}{GL}
\newcommand{\ZZ}{\mathbb{Z}}
\newcommand{\RR}{\mathbb{R}}
\newcommand{\PP}{\mathbb{P}}
\newcommand{\QQ}{\mathbb{Q}}
\newcommand{\CC}{\mathbb{C}}
\newcommand{\CP}{\mathbb{CP}}
\newcommand{\Gammaab}{\Gamma(\alpha,\beta)}
\newcommand{\abs}[1]{\lvert #1 \rvert}

\title{Resolving Thinness of Degree-Six Symplectic Hypergeometric Monodromy Groups}

\author{
  J. Maxwell Riestenberg\thanks{Max Planck Institute for Mathematics in the Sciences. Email:
		\href{mailto:riestenberg@mis.mpg.de} {\nolinkurl {riestenberg@mis.mpg.de}}.}
\and
  Diaaeldin Taha\thanks{Max Planck Institute for Mathematics in the Sciences. Email:
		\href{mailto:taha@mis.mpg.de} {\nolinkurl {taha@mis.mpg.de}}.} 
\and
  Steve Trettel\thanks{University of San Francisco. Email:
		\href{mailto:strettel@usfca.edu} {\nolinkurl {strettel@usfca.edu}}.} 
\and
  Anna Wienhard\thanks{Max Planck Institute for Mathematics in the Sciences. Email:
		\href{mailto:wienhard@mis.mpg.de} {\nolinkurl {wienhard@mis.mpg.de}}.} 
        \thanks{Center for Scalable Data Analytics and Artificial Intelligence, Leipzig, Germany}
}

\date{}

\begin{document}

\maketitle

\begin{abstract}
We prove arithmeticity for two degree-six symplectic hypergeometric monodromy groups and thinness for a third case.
These are the three open cases called C-32, C-47, and C-55 in the paper [BDN25] by Bajpai--Dona--Nitsche. This hence settles thinness of primitive integral symplectic hypergeometric monodromy groups of degree-six, according to the tabulation by Bajpai--Dona--Singh--Singh [BDSS21]. 
The arithmeticity certificates were found with AlphaEvolve and the thinness certificate is a ping pong table found by a harness that was configured to run ChatGPT 5.5.
These certificates are then independently verified with exact matrix arithmetic over Q using a computer. 
We also include illustrations of limit sets of several degree-six symplectic hypergeometric monodromy groups.
\end{abstract}

\section{Introduction}

It is easy to give a subgroup of a matrix group such as $\mathrm{GL}(n,\mathbb{R})$, just write down some invertible matrices and consider the group $\Gamma$ they generate.
When you define a group this way, it is however quite difficult in general to obtain information about the group $\Gamma$. 
Is it discrete? And if yes, how big is it? What is its Zariski closure? Is it a lattice? 

One way to ensure discreteness is to pick matrices as generators which lie in an arithmetic lattice itself, for example in the group of integer matrices $\mathrm{SL}(n,\mathbb{Z})$ inside of $\mathrm{SL}(n,\mathbb{R})$. 
Then the main question becomes whether $\Gamma$ is of finite index in $\mathrm{SL}(n,\mathbb{Z})$ and thus itself a lattice, or of infinite index and thus a much smaller discrete subgroup. 
Such discrete subgroups of arithmetic lattices in $\mathrm{SL}(n,\mathbb{R})$ or more general semisimple matrix groups, which are Zariski dense, but of infinite index are called ``thin groups.'' 

The name ``thin groups'' was introduced by Peter Sarnak more than twenty years ago, motivated by the affine sieve methods developed by Bourgain-Gamburd-Sarnak, which allows to draw number theoretic conclusions for such groups, which were out of reach before. 

Particular nice families of discrete subgroups in arithmetic lattices arise from monodromy groups, see for example the discussion in Section~3.5. in \cite{Sarnak2014NotesThin}. 
They are often given by explicit generating matrices, and arise in interesting families. 

The paper concerns certain monodromy representations associated to the classical hypergeometric differential equation, which received a lot of attention over the past 20 years. These are representations of the fundamental group of $\mathbb{C}\mathbb{P}^1 \backslash \{ 0, 1, \infty\}$ into $\mathrm{GL}_n(\mathbb{R})$, which are generated by the local monodromies around $0$, $1$, and $\infty$ given by matrices $A$, $B$, $C = A^{-1}B$. 
See 
\cite{Levelt1961HypergeometricFunctions,BeukersHeckman1989hypergeometric} for details. 

A special interest has been in such hypergeometric monodromy groups which preserve a symplectic form $\Omega$ and are Zariski dense in $\mathrm{Sp}_{\Omega}(n,\mathbb{Z})$.

The symplectic degree-four cases, which are connected to Calabi--Yau threefolds became an early test family.  
Brav--Thomas proved thinness for seven of them \cite{BravThomas2014Thin}, while arithmeticity results of Singh--Venkataramana and Singh settled the complementary cases \cite{SinghVenkataramana2014Arithmeticity,Singh2015Four}.  
In degree-six, Bajpai--Dona--Singh--Singh studied the corresponding symplectic hypergeometric groups, motivated in part by a question of Katz about whether the maximally unipotent degree-six family follows a similar pattern \cite{BajpaiDonaSinghSingh2021degreesix}.  
Bajpai--Dona--Nitsche later settled many further degree-six cases, including many thin cases, by computer-assisted ping-pong \cite{BajpaiDonaNitsche2025Thin}.  
In their work \cite[Table 3]{BajpaiDonaNitsche2025Thin} they list the three remaining degree-six symplectic cases, C-32, C-47, and C-55, for which neither arithmeticity nor thinness was known.

This note resolves the three remaining cases. 
We show that the two open cases called C-47 and C-55 are arithmetic, and that the C-32 is a thin virtually free group.
We also include visualizations of the limit sets for all three cases, C-32, C-47, and C-55.

To prove arithmeticity of the two groups C-47 and C-55, we use the arithmeticity criterion of Bajpai--Dona--Nitsche \cite{BajpaiDonaNitsche2026Arithmetic} building on previous arithmeticity criteria of Venkataramana, Venkataramana-Singh \cite{Venkataramana1987Zariskidensesubgroups, SinghVenkataramana2014Arithmeticity,Singh2015Four}. 
Their criterion reduces the problem to  exhibiting explicit words in the  generators $A$ and $B$ of the hypergeometric monodromy whose conjugates of the standard rank-one unipotent give a  pair of commuting transvections.  
The words were found with AlphaEvolve \cite{NovikovEtAl2025AlphaEvolve,georgiev2025mathematical}.  
Once the words are fixed, the verification is deterministic and uses only exact matrix arithmetic over $\QQ$.

After initially finding the proof of arithmeticity of the groups C-47 and C-55, we drew visualizations of the limit sets of various degree-six hypergeometric monodromy groups, see \Cref{sec:limitset}.
This led us to conjecture that the remaining group C-32 was thin.
With the help of ChatGPT 5.5 \cite{chatgpt55}, we then found a precise ping pong table demonstrating that group is virtually free, which proves thinness. 

\paragraph{Use of AI.} The arithmeticity certificates for C-47 and C-55 were found using AlphaEvolve. The polyhedral cone used to prove thinness of C-32 was found using an unpublished harness configured, at the time of the search in May 2026, to run ChatGPT 5.5. The resulting certificates are verified independently by exact arithmetic, as described below. The authors independently wrote the mathematical arguments, computations, and conclusions presented in this paper.

\section{Hypergeometric monodromy groups}

Let $\alpha,\beta\in\mathbb{C}^n$. 
The hypergeometric differential equation may be written as 
    \[ [z(\theta + \alpha_1) \cdots (\theta +\alpha_n) - (\theta +\beta_1-1)\cdots(\theta+\beta_n-1)]u(z) = 0 \]
where $\theta = z \frac{\mathrm{d}}{\mathrm{d}z}$. 
This differential equation is regular on $\CP^1$ away from $0,1,\infty$, so there is a monodromy representation
    \[ \rho_{\alpha,\beta} \colon \pi_1(\CP^1 \setminus \{0,1,\infty\}) \to \GL_n(\mathbb{C}) \]
on the space of solutions, well-defined up to conjugation.
When any difference $\alpha_j - \beta_k$ for some $j,k$ is integral, the corresponding representation is reducible \cite[Proposition 2.7]{BeukersHeckman1989hypergeometric}. 
Otherwise, 
Levelt \cite{Levelt1961HypergeometricFunctions} proved that the image is generated by the companion matrices associated to the polynomials 
\[
  f_\alpha(x)=\prod_{j=1}^n\bigl(x-\exp(2\pi i\alpha_j)\bigr),\qquad
  g_\beta(x)=\prod_{j=1}^n\bigl(x-\exp(2\pi i\beta_j)\bigr).
\]
We remind the reader that, for example, for a degree six monic polynomial
\[
  h(x)=x^{6}+c_5x^5+c_4x^4+c_3x^3+c_2x^2+c_1x+c_0,
\]
the companion matrix is
\[
C(h)=
\begin{pmatrix}
0&0&0&0&0&-c_0\\
1&0&0&0&0&-c_1\\
0&1&0&0&0&-c_2\\
0&0&1&0&0&-c_3\\
0&0&0&1&0&-c_4\\
0&0&0&0&1&-c_5
\end{pmatrix}.
\]

Beukers and Heckman subsequently classified the Zariski closures of hypergeometric monodromy groups. 
Following \cite[Definition 5.1]{BeukersHeckman1989hypergeometric}, a linear group $\Gamma$ acting irreducibly on $\mathbb{C}^n$ is called imprimitive when its action on $\mathbb{C}^n$ permutes a non-trivial direct sum decomposition, and \emph{primitive} otherwise. 
A linear group is called strongly irreducible when each finite-index subgroup acts irreducibly on $\mathbb{C}^n$. 
Selberg's lemma asserts that a finitely generated linear group admits a finite-index torsion-free subgroup.
We observe that, as a consequence, a finitely generated linear group with reductive Zariski closure is primitive exactly when it is strongly irreducible.

We are interested in integral hypergeometric monodromy groups, which is to say hypergeometric monodromy groups which may be conjugated into $\mathrm{GL}_n(\ZZ)$. 
If we further impose that the monodromy group is primitive, then the corresponding polynomials $f_\alpha,g_\beta$ are products of cyclotomic polynomials without common factors. Moreover, such hypergeometric monodromy groups are self-dual, hence preserve some orthogonal or symplectic form. 
In fact, Beukers-Heckman prove that the Zariski closure of such a group is either $\Sp(n,\CC)$ (and $n$ is even) or $\mathrm{O}(n,\CC)$ depending on the sign of $f(0)/g(0) \in \{-1,1\}$ \cite{BeukersHeckman1989hypergeometric}. 

The upshot is that there are finitely many symplectic integral primitive hypergeometric monodromy groups in degree six up to conjugation, and these have been counted and tabulated by \cite{BajpaiDonaSinghSingh2021degreesix}: they obtained 458 pairs of degree six polynomials $f,g$ that are products of cyclotomic polynomials, without common roots, form a primitive pair, and satisfy $f(0)=g(0)=1$, up to scalar shift. 
Subsequent work has determined that each such group is either arithmetic (i.e.\ a lattice) or a virtually free group (in particular, thin), leaving 3 open cases. 
We resolve these, proving that two of the cases are arithmetic and the last case is virtually free.

\section{C-47 and C-55 are arithmetic}
\label{sec:c-47-c-55-arithmetic}

Our proof of arithmeticity uses the following criterion of Bajpai--Dona--Nitsche
\cite[Lemma 1]{BajpaiDonaNitsche2026Arithmetic}, which builds on work of
Singh--Venkataramana \cite{SinghVenkataramana2014Arithmeticity} and
Venkataramana \cite{Venkataramana1987Zariskidensesubgroups}.

\begin{lemma}[Bajpai--Dona--Nitsche]\label{lem:bdn}
    Let $n\geq 2$, let $\Omega$ be a nondegenerate symplectic form on $\QQ^{2n}$
    which is integral on $\ZZ^{2n}$, and let $\Gamma<\Sp_{\Omega}(\ZZ)$ be Zariski-dense. 
    Then $\Gamma$ has finite index in $\Sp_{\Omega}(\ZZ)$ if and only if $\Gamma$ contains two transvections
    \[
      X_i=1+\lambda_i x_i\Omega(x_i,\cdot),\qquad
      \lambda_i\in\QQ^\times,\quad i=1,2,
    \]
    whose directions $x_1,x_2$ are linearly independent and
    $\Omega$-orthogonal, i.e.\ $\Omega(x_1,x_2)=0$.
\end{lemma}

We recall the definitions of the two groups labelled C-47 and C-55 in \cite[Table 3]{BajpaiDonaNitsche2025Thin}. 
For
C-47, the parameter multisets are
\[
  \alpha_{47}= \left(0,0,\frac{1}{5},\frac{2}{5},\frac{3}{5},\frac{4}{5} \right),\qquad
  \beta_{47} = \left( \frac{1}{2},\frac{1}{2},\frac{1}{3},\frac{1}{3},\frac{2}{3},\frac{2}{3} \right),
\]
and the corresponding polynomials are
\[
  f_{47}=f_{\alpha_{47}}=x^6-x^5-x+1,\qquad
  g_{47}=g_{\beta_{47}}=x^6+4x^5+8x^4+10x^3+8x^2+4x+1.
\]
For C-55, the parameter multisets are
\[
  \alpha_{55}=\left(0,0, \frac{1}{8},\frac{3}{8},\frac{5}{8},\frac{7}{8} \right),\qquad
  \beta_{55}= \left(\frac{1}{2},\frac{1}{2},\frac{1}{12},\frac{5}{12},\frac{7}{12},\frac{11}{12}\right),
\]
and the corresponding polynomials are
\[
  f_{55}=f_{\alpha_{55}}=x^6-2x^5+x^4+x^2-2x+1,\qquad
  g_{55}=g_{\beta_{55}}=x^6+2x^5-2x^3+2x+1.
\]
The associated monodromy groups are 
\[
  \Gamma_{47}=\langle C(f_{47}),C(g_{47})\rangle,\qquad
  \Gamma_{55}=\langle C(f_{55}),C(g_{55})\rangle.
\]
The integral symplectic form preserved by $\Gamma_{47}$ is
\[
\Omega_{47}=\begin{pmatrix}
0&29&-50&51&-28&1\\
-29&0&29&-50&51&-28\\
50&-29&0&29&-50&51\\
-51&50&-29&0&29&-50\\
28&-51&50&-29&0&29\\
-1&28&-51&50&-29&0
\end{pmatrix}
\]
and the integral symplectic form preserved by $\Gamma_{55}$ is
\[
\Omega_{55}=\begin{pmatrix}
0&1&6&3&4&5\\
-1&0&1&6&3&4\\
-6&-1&0&1&6&3\\
-3&-6&-1&0&1&6\\
-4&-3&-6&-1&0&1\\
-5&-4&-3&-6&-1&0
\end{pmatrix}.
\]

\begin{theorem}
\label{thm:main}
    The hypergeometric monodromy groups labelled C-47 and C-55 in
    \cite[Table 3]{BajpaiDonaNitsche2025Thin} have finite index in their ambient
    integral symplectic groups.  
    In particular, both groups are arithmetic.
\end{theorem}

We now prove that C-47 and C-55 are arithmetic by verifying explicit certificates found with AlphaEvolve.  
Once the witness words are fixed, the verification is deterministic and uses only exact matrix arithmetic over $\QQ$.

In each case, write $A=C(f)$ and $B=C(g)$, and set
$a=A^{-1}$ and $b=B^{-1}$. In the following proof, we read the words $w_{47}$ and $w_{55}$ in $A$, $B$, $a$, and $b$ as the product
of these matrices.

\begin{proof}
First consider C-47.  
Let $w_{47}$ be the following freely reduced word of length $93$:
\begin{lstlisting}
aabaBaBBaBaBAAAAAAABaBaBaBaBaabAbAAABBBBBBaaBaBAbbbbbbAAbbbbbbAAbbbbbbaabbbbbbAAbbbbbb
aabbbbb
\end{lstlisting}
Let $\gamma\in\Gamma_{47}$ be the product $w_{47}$,
and set $T=A^{-1}B$. A direct computation over $\ZZ$ gives
\[
  \operatorname{im}(T-I)=\QQ x_1,\qquad
  \operatorname{im}(\gamma T\gamma^{-1}-I)=\QQ x_2,
\]
where
\[
  x_1=(-5,-8,-10,-8,-5,0),
\]
and
\[
  x_2=
  (1002449,1096627,458382,150715,-38701,0).
\]
The vectors $x_1$ and $x_2$ are linearly independent.
Moreover,
\[
  \Omega_{47}(x_1,x_2)=0.
\]
The same exact computation gives that $T$ and $\gamma T\gamma^{-1}$ are
rank-one unipotents and that
\[
  T \, \gamma T \gamma^{-1} = \gamma T \gamma^{-1} \, T .
\]
Thus $T$ and $\gamma T\gamma^{-1}$ are commuting transvections whose
directions are linearly independent and $\Omega_{47}$-orthogonal.

\medskip

Now consider C-55.  
Let $w_{55}$ be the following freely reduced word of length $49$:
\begin{lstlisting}
BaabAABABaBABaaaaabaaaabaaaabaaaabaaaabaaaabaaaab
\end{lstlisting}
Let $\gamma\in\Gamma_{55}$ be the product $w_{55}$, and set $T=A^{-1}B$.

A direct computation over $\ZZ$ gives
\[
  \operatorname{im}(T-I)=\QQ x_1,\qquad
  \operatorname{im}(\gamma T\gamma^{-1}-I)=\QQ x_2,
\]
where
\[
  x_1=(-4,1,2,1,-4,0),
\]
and
\[
  x_2=
  (3865,-25966,-12448,22278,29671,0).
\]
The vectors $x_1$ and $x_2$ are linearly independent. 
Moreover,
\[
  \Omega_{55}(x_1,x_2)=0.
\]
The same exact computation gives that $T$ and $\gamma T\gamma^{-1}$ are
rank-one unipotents and that
\[
  T \, \gamma T \gamma^{-1} = \gamma T \gamma^{-1} \, T . 
\]
Thus $T$ and $\gamma T\gamma^{-1}$ are commuting transvections whose
directions are linearly independent and $\Omega_{55}$-orthogonal.

In each case, the form $\Omega$ is integral and nondegenerate, and an easy check shows that the corresponding group $\Gamma=\langle A,B\rangle$ lies in $\Sp_\Omega(\ZZ)$. 
The Zariski-density hypothesis in Lemma \ref{lem:bdn} is supplied by the Beukers--Heckman classification, as recalled above.  
Lemma \ref{lem:bdn} therefore applies to C-47 and C-55, and shows that each group has finite index in its ambient integral symplectic group.
\end{proof}

The accompanying
scripts
\footnote{Available as as \texttt{verify\_c47\_witness\_annotated.py} and \texttt{verify\_c55\_witness\_annotated.py} at \url{https://github.com/ditahd/Sp-6-Certificates}}
reproduce the exact computations used in the proof. 
Precisely, they each reconstruct $A$, $B$, the symplectic form, the witness word, and its inverse. 
They verify over $\ZZ$ and $\QQ$ that $A$ and $B$ preserve the displayed form, that $T$ and $\gamma T\gamma^{-1}$ are rank-one unipotents, that the two elements commute, and that their directions are linearly independent and $\Omega$-orthogonal.

\section{C-32 is thin}
\label{sec:c-32-thin}

Our proof of thinness uses the following ping pong criterion of Bajpai--Dona--Nitsche
\cite[Theorem~7]{BajpaiDonaNitsche2025Thin}, which builds on work of
Brav--Thomas \cite{BravThomas2014Thin} and on the ping pong lemma of
Lyndon--Schupp \cite[Prop.~III.12.4]{LyndonSchupp2001}.

\begin{lemma}[Bajpai--Dona--Nitsche]
\label{lem:pingpong}
Let $G$ be a group generated by two subgroups $G_1$ and $G_2$ whose intersection $Z$ has index $> 2$ in $G_1$ or $G_2$. Suppose that $G$ acts on a set $W$ containing disjoint nonempty subsets $X,Y$ such that $(G_1\setminus Z)Y\subseteq X$ and $(G_2\setminus Z)\,X\subseteq Y$ with $ZX\subseteq X$ and $ZY\subseteq Y$. Then $G=G_1\ast_Z G_2$.
\end{lemma}

For a hypergeometric group generated by the companion matrices $A=C(f)$ and $B=C(g)$,
Bajpai--Dona--Nitsche apply this with $T=BA^{-1}$, $G_1=\langle\pm T\rangle$,
$G_2=\langle B\rangle$ and $Z=\{\pm I\}$ (or $G_1=\langle T\rangle$, $Z=\{I\}$
if $-I\notin\langle B\rangle$). Let $P_\pi$ be the permutation matrix sending
$e_i$ to $e_{\pi(i)}$, where $\pi=(1\,6)(2\,5)(3\,4)$, and set $E=BP_\pi$; then
\[ E^2=I,\qquad EBE^{-1}=B^{-1},\qquad ETE^{-1}=T^{-1}. \]
Put $\eta=\min\{k>0:(B^k-I)^6=0\}$ and, for open convex cones $C,D\subset\RR^6$,
\[
  X^+=C\cup -C,\quad X^-=EX^+,\quad
  Y^+=\bigcup\{\pm B^iD : 1\le i\le\eta,\ B^i\ne\pm I\},\quad Y^-=EY^+,
\]
with $X=X^+\cup X^-$ and $Y=Y^+\cup Y^-$. Then $ZX\subseteq X$ and
$ZY\subseteq Y$ are automatic, and the remaining hypotheses of
\Cref{lem:pingpong} hold provided
\begin{enumerate}
  \item[(1)] $X\neq\emptyset$;
  \item[(2)] $X\cap Y=\emptyset$;
  \item[(3)] $B^iX\subseteq Y$ for every $i\in\ZZ$ with $B^i\neq\pm I$;
  \item[(4)] $T^iY\subseteq X$ for every $i\in\ZZ\setminus\{0\}$.
\end{enumerate}
Condition (4) may be replaced by the two finite conditions
\begin{enumerate}
  \item[(4a)] $T^{-1}(Y\cup X^+)\subseteq X^+$;
  \item[(4b)] $T(Y\cup X^-)\subseteq X^-$.
\end{enumerate}
If $B$ has finite order, (3) is a finite check and one takes $D=C$
\cite[\S3.1]{BajpaiDonaNitsche2025Thin}.

We recall the definition of the group labelled C-32 in \cite[Table 3]{BajpaiDonaNitsche2025Thin}. Its parameter multisets are
\[ \alpha_{32} = \left( 0,0,0,0,\frac{1}{6},\frac{5}{6}\right) , \text{ and } \beta_{32} = \left( \frac{1}{4},\frac{3}{4},\frac{1}{12},\frac{5}{12},\frac{7}{12},\frac{11}{12}\right) .\]

This leads to the polynomials
\begin{equation*}
  f_{32}(x)=(x-1)^4\Phi_6(x)
  =x^6-5x^5+11x^4-14x^3+11x^2-5x+1
\end{equation*}
and
\begin{equation*}
  g_{32}(x)=\Phi_4(x)\Phi_{12}(x)=x^6+1.
\end{equation*}
Let $A_0,B_0$ be their companion matrices:
\begin{equation*}
A_0 = C(f_{32}) =
\begin{pmatrix}
0&0&0&0&0&-1\\
1&0&0&0&0&5\\
0&1&0&0&0&-11\\
0&0&1&0&0&14\\
0&0&0&1&0&-11\\
0&0&0&0&1&5
\end{pmatrix},
\qquad
B_0 = C(g_{32}) =
\begin{pmatrix}
0&0&0&0&0&-1\\
1&0&0&0&0&0\\
0&1&0&0&0&0\\
0&0&1&0&0&0\\
0&0&0&1&0&0\\
0&0&0&0&1&0
\end{pmatrix}.
\end{equation*}
So that $\Gamma_{32} = \langle A_0,B_0 \rangle$. The integral symplectic form preserved by $\Gamma_{32}$ is
\[
\Omega_{32}=\begin{pmatrix}
0&13&22&25&22&13\\
-13&0&13&22&25&22\\
-22&-13&0&13&22&25\\
-25&-22&-13&0&13&22\\
-22&-25&-22&-13&0&13\\
-13&-22&-25&-22&-13&0
\end{pmatrix}.
\]

\begin{theorem}
    The group generated by the image of $\Gamma_{32}$ in $\PSp(6,\RR)$ is isomorphic to $\ZZ \ast \ZZ/6$.
    In particular, $\Gamma_{32}$ is a thin subgroup of $\Sp_\Omega(\ZZ)$. 
\end{theorem}

\begin{proof}
Note that the product
\begin{equation*}
  T_0=B_0A_0^{-1} = \left(
\begin{array}{cccccc}
 1 & 0 & 0 & 0 & 0 & 0 \\
 5 & 1 & 0 & 0 & 0 & 0 \\
 -11 & 0 & 1 & 0 & 0 & 0 \\
 14 & 0 & 0 & 1 & 0 & 0 \\
 -11 & 0 & 0 & 0 & 1 & 0 \\
 5 & 0 & 0 & 0 & 0 & 1 \\
\end{array}
\right).
\end{equation*}
is a transvection.
As in \Cref{lem:pingpong}, we use $T_0=B_0A_0^{-1}$ here rather than $A_0^{-1}B_0$ as in \Cref{sec:c-47-c-55-arithmetic}; the two are conjugate, since $A_0(A_0^{-1}B_0)A_0^{-1}=B_0A_0^{-1}$.
Let
\begin{equation*}
  v=(T_0-I)e_0=(0,5,-11,14,-11,5)^t
\end{equation*}
and consider the basis obtained by the orbit of $(-B_0)$ on $v$: 
\begin{equation*}
  P=\bigl[v,\,-B_0v,\,B_0^2v,\,-B_0^3v,\,B_0^4v,\,-B_0^5v\bigr].
\end{equation*}
It will be easier to describe the ping pong table in this basis.
Conjugating by $P$ gives generators 
    \[ S=P^{-1}B_0P = \left(
\begin{array}{cccccc}
 0 & 0 & 0 & 0 & 0 & 1 \\
 -1 & 0 & 0 & 0 & 0 & 0 \\
 0 & -1 & 0 & 0 & 0 & 0 \\
 0 & 0 & -1 & 0 & 0 & 0 \\
 0 & 0 & 0 & -1 & 0 & 0 \\
 0 & 0 & 0 & 0 & -1 & 0 \\
\end{array}
\right) \text{and } T=P^{-1}T_0P  = \left(
\begin{array}{cccccc}
 1 & 5 & 11 & 14 & 11 & 5 \\
 0 & 1 & 0 & 0 & 0 & 0 \\
 0 & 0 & 1 & 0 & 0 & 0 \\
 0 & 0 & 0 & 1 & 0 & 0 \\
 0 & 0 & 0 & 0 & 1 & 0 \\
 0 & 0 & 0 & 0 & 0 & 1 \\
\end{array}
\right) \]
There is also an involution $E$ given in this basis by
\begin{equation*}
E=
\begin{pmatrix}
1&0&0&0&0&0\\
0&0&0&0&0&-1\\
0&0&0&0&-1&0\\
0&0&0&-1&0&0\\
0&0&-1&0&0&0\\
0&-1&0&0&0&0
\end{pmatrix}.
\end{equation*}
It satisfies
\begin{equation*}
  E^2=I,\qquad ESE=S^{-1},\qquad ETE=T^{-1}.
\end{equation*}

In order to describe the ping pong table, we need the $33 \times 6$ matrix $H$ which is described below in Appendix \ref{app:c32-cone}.
It determines the cone 
\begin{equation}\label{eqn: H-rep of K}
    K = \left\{ y \in \RR^6 : Hy \ge 0 \right\}
\end{equation}
where $Hy \ge 0$ means each of the 33 entries of $Hy$ is nonnegative.
Projectivizing, we obtain domains
    \[ X^+ =
\PP(K^\circ)
, \quad X^- = E X^+, \quad X = X^+ \cup X^- ,\]
and 
    \[ Y = \bigcup_{k=1}^5 S^k X .\]

\begin{lemma}\label{lem:cone containment}
    For each $1 \le k \le 5$ and $a \in \{0,1\}$ it holds that 
        \[ (-1)^{k+1} T^{-1} S^k E^a K \subseteq K .\]
    Moreover, 
        \[ T^{-1} K \subseteq K .\]
\end{lemma}

\begin{proof}
    This proof is a computation performed in the code.\footnote{Available as \texttt{verify\_c32\_thinness\_annotated.py} at \url{https://github.com/ditahd/Sp-6-Certificates}}. 
    A key step is describing the convex cone $K$ in \ref{eqn: H-rep of K} as the convex hull of its extremal rays, rather than as an intersection of half-spaces.
    This is done in exact rational arithmetic via the cdd library \cite{cddlib}, which is an implementation of the double description method \cite{MRTT1953,FukudaProdon1996}.
\end{proof}

All the maps in \Cref{lem:cone containment} are invertible. So, their inclusions for $K$ also hold for $K^\circ$. Moreover, $ESE=S^{-1}$ and $S^6=-I$ imply $EY=Y$ projectively. Conjugation by $E$ therefore gives the corresponding inclusions for $T$.

It follows from \Cref{lem:cone containment} that 
    \[ T Y \subseteq X, \quad T^{-1} Y \subseteq X, \quad
    T^{-1}(Y\cup X^+)\subseteq X^+,\quad T(Y\cup X^-)\subseteq X^-
    \]
and 
    \[ S^k X \subseteq Y \text{ for all } 1 \le k \le 5.\]
Iterating these inclusions gives $T^nY\subseteq X$ for every nonzero
integer $n$. The vector
    \[
    y=(47,-1,-1,-1,-1,-1)^t
    \]
satisfies $Hy>0$,
so $X$ and $Y$ are nonempty. The first ten rows of $H$, listed in
Appendix~\ref{app:c32-cone}, imply that $|y_0|>|y_i|$ for every
$i>0$ on $X^+$. Since $E$ preserves $|y_0|$ and permutes
$|y_1|,\ldots,|y_5|$, the same inequalities hold on $X^-$, hence
on $X$. For $1\le k\le5$, the map $S^k$ moves the coordinate
of largest absolute value from position $0$ to position $k$.
Thus $X$ is disjoint from each $S^kX$, and the definition of $Y$
gives $X\cap Y=\emptyset$. Finally, $S$ has order six on
projective space, while $T$ has infinite order there. We can then apply the ping pong lemma to the subgroup of $\PSp(6,\RR)$ generated by $S$ and $T$.
Precisely, after the change of basis $P$ we are in the setting of \Cref{lem:pingpong}, with $S$ in the role of $B$. In $\RR\PP^5$, the set $X^+=\PP(K^\circ)$ is the image of $C\cup -C$ for $C=K^\circ$, and $Y$ is the image of $Y^+\cup Y^-$ for $D=C$, since $S$ has finite order. The nonemptiness of $X$, the disjointness $X\cap Y=\emptyset$, and the inclusions above are conditions (1), (2), (3), (4a) and (4b). Moreover, $Z=\{\pm I\}$ has index $6>2$ in $\langle S\rangle$.
This shows that the subgroup of $\PSp(6,\RR)$ generated by $S$ and $T$ is isomorphic to $\ZZ \ast \ZZ/6$, which implies that $\Gamma_{32} \cong
(\ZZ\times\ZZ/2)\ast_{\ZZ/2}\ZZ/12
$.
This implies that $\Gamma_{32}$ is thin: $\langle\pm T_0\rangle$ and $\langle B_0\rangle$ are abelian, so their amalgamated product over the finite group $\{\pm I\}$ has the Haagerup property, whereas a finite-index subgroup of $\Sp_{\Omega_{32}}(\ZZ)$ would be a lattice in $\Sp(6,\RR)$ and hence have Kazhdan's property (T). An infinite discrete group cannot have both.
\end{proof}
    
While not used in the proof, we can briefly comment on the structure of these domains.
There is a nice decomposition of $\RR\PP^5$ into $6$ hypercubes according to the largest homogeneous coordinate in absolute value:
    \[ \Delta_i = \left\{ [x_0:x_1:x_2:x_3:x_4:x_5] : \abs{ x_i } \ge \abs{x_j} \text{ for } j \ne i \right\} .\]
In the basis we use, $S$ permutes the domains $\Delta_i$, $0 \le i \le 5$, and $T$ almost takes $\Delta_i$, $1\le i \le 5$ to $\Delta_0$. 
Moreover, $E$ preserves $\Delta_0$, acting with order 2 and a fixed point set of dimension 2. 
Inspecting the columns of $H$, one finds that $X^+$, hence $X$, is a subset of $\Delta_0$. 
In visualizations, $X$ appears to be ``almost all" of $\Delta_0$. 

\section{Limit set visualizations}\label{sec:limitset}

The proximal limit set $\Lambda(\Gammaab)$ of $\Gammaab = \langle A, B \rangle$
is the closure in $\mathbb{RP}^5$ of the attracting fixed lines of all
loxodromic elements of $\Gammaab$, and has visibly different qualitative
signatures in the two regimes: a proper compact subset (often fractal) in
the thin case versus the entire $\mathbb{RP}^5$ in the arithmetic case.  
We accompany the algebraic results with visualizations of these limit sets.

\paragraph{How we compute it.}

We approximate $\Lambda(\Gammaab)$ by a finite \emph{partial
orbit}: starting from one attracting fixed line $\xi_+ \in \Lambda(\Gammaab)$ of a loxodromic $\eta$, we
apply to it the set of freely reduced words of length at most $N$ in
$\{A^{\pm 1}, B^{\pm 1}\}$.  
For the loxodromic element we take $\eta = TBT$, with $T = A^{-1}B$ as previously.  In all five examples $\eta$ has a real, simple,
dominant top eigenvalue $\lambda_1$ with spectral gap $|\lambda_1/\lambda_2| \geq 10$,
so the iteration
\[
  v_{k+1} \;=\; \frac{\eta\, v_k}{\|\eta\, v_k\|}
\]
starting from a generic seed converges projectively to $\xi_+$ to working
precision in under $30$ steps.

We walk through the set of words of length $\leq N$ inductively and compute
the corresponding vectors in the orbit $\mathcal{O}_N$ at each step. For a word
$w = w'\, s$ the representative $w \cdot \xi_+ \in \mathbb{R}^6$ is
obtained from the parent's representative $w' \cdot \xi_+$ by a single
matrix-vector product with the generator $s$, then rescaled to unit norm
(preserving the projective class). 
In the images, we color points by the final generator applied with $a=$red, $A=$yellow, and $b=$green, $B=$blue.
Our visualizations use words of length up to 21, yielding orbits of approximately 20 billion points.

\paragraph{Projection $\mathbb{RP}^5 \dashrightarrow \mathbb{R}^3$.}
Fix $\ell \in (\mathbb{R}^6)^*$ and project onto the affine chart
$H_\ell = \{v : \ell(v) = 1\}$, discarding representatives with
$|\ell(v)| < 10^{-3}$ near infinity in the patch.  Inside $H_\ell$ we compute the principal components
of the centered cloud $\{v/\ell(v) : v \in \mathcal{O}_N\}$ and plot its
projection onto the top three.  Using principal components ensures that the
three plotted axes capture the directions of greatest variance in the
cloud, so that an orbit that genuinely spreads in more than three
dimensions is not artificially flattened by the projection.

\paragraph{The figures.}
We compare partial orbits of thin and arithmetic groups and illustrate the ping-pong domains for C-32.
Figure~\ref{fig:calibration} fixes
the visual templates of arithmetic and thin, showing the partial orbits
for two cases classified by \cite{BajpaiDonaNitsche2025Thin}: in their enumeration, A-1 (thin) and A-15 (arithmetic).  The proximal limit set of A-1 lies near a thin curve whereas even just the first million points for A-15 spread out in many directions.

Figure~\ref{fig:arithmetic} shows the
partial orbits for C-47 and C-55, the two cases proved arithmetic in
Theorem~\ref{thm:main}; both clouds can also be seen to be more complex than that of A-1.  Figure~\ref{fig:c32} shows the partial orbit for C-32 appearing as a curve qualitatively closer to A-1 than to A-15 or C-47, C-55 images. Finally, figure~\ref{fig:domains} complements the limit set images by illustrating the geometry of the ping-pong construction in Section~\ref{sec:c-32-thin}. It shows the projectivized cone $\PP(K)$, its five nontrivial translates under $S$, and their images under $T^{-1}$, which lie inside the original domain by Lemma~\ref{lem:cone containment}.

\begin{figure}[t]
  \centering
  \includegraphics[width=0.49\linewidth]{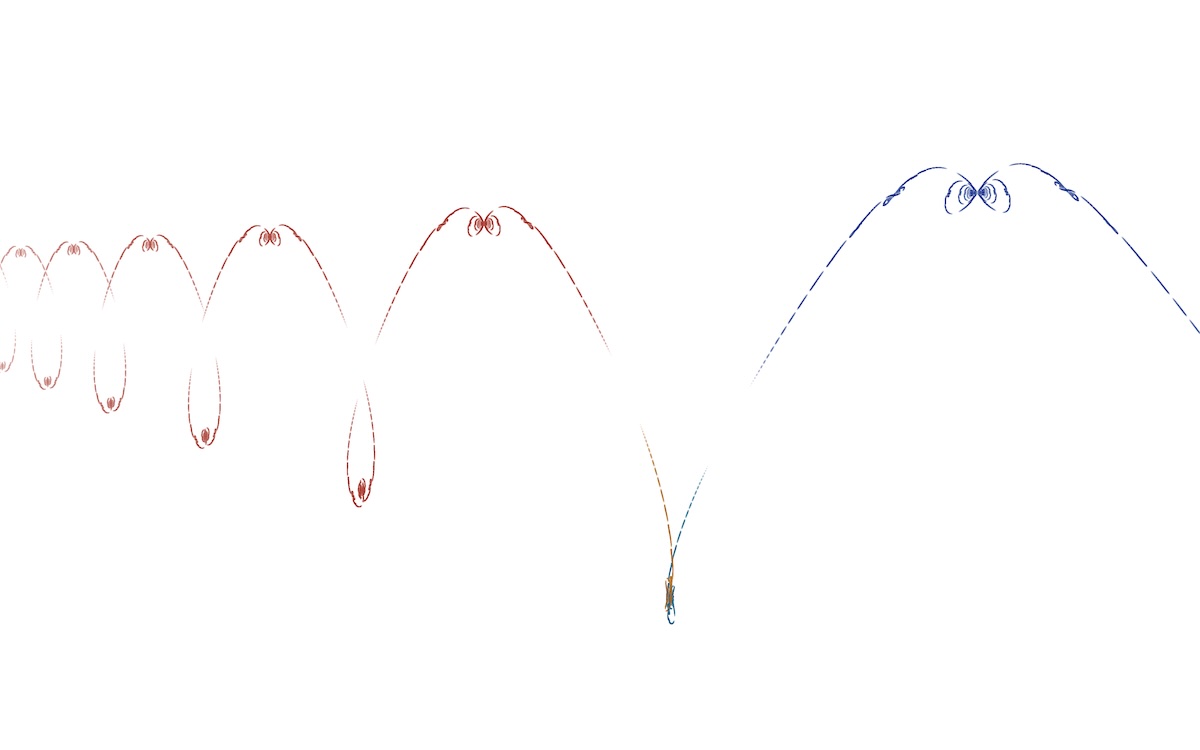}\hfill
  \includegraphics[width=0.49\linewidth]{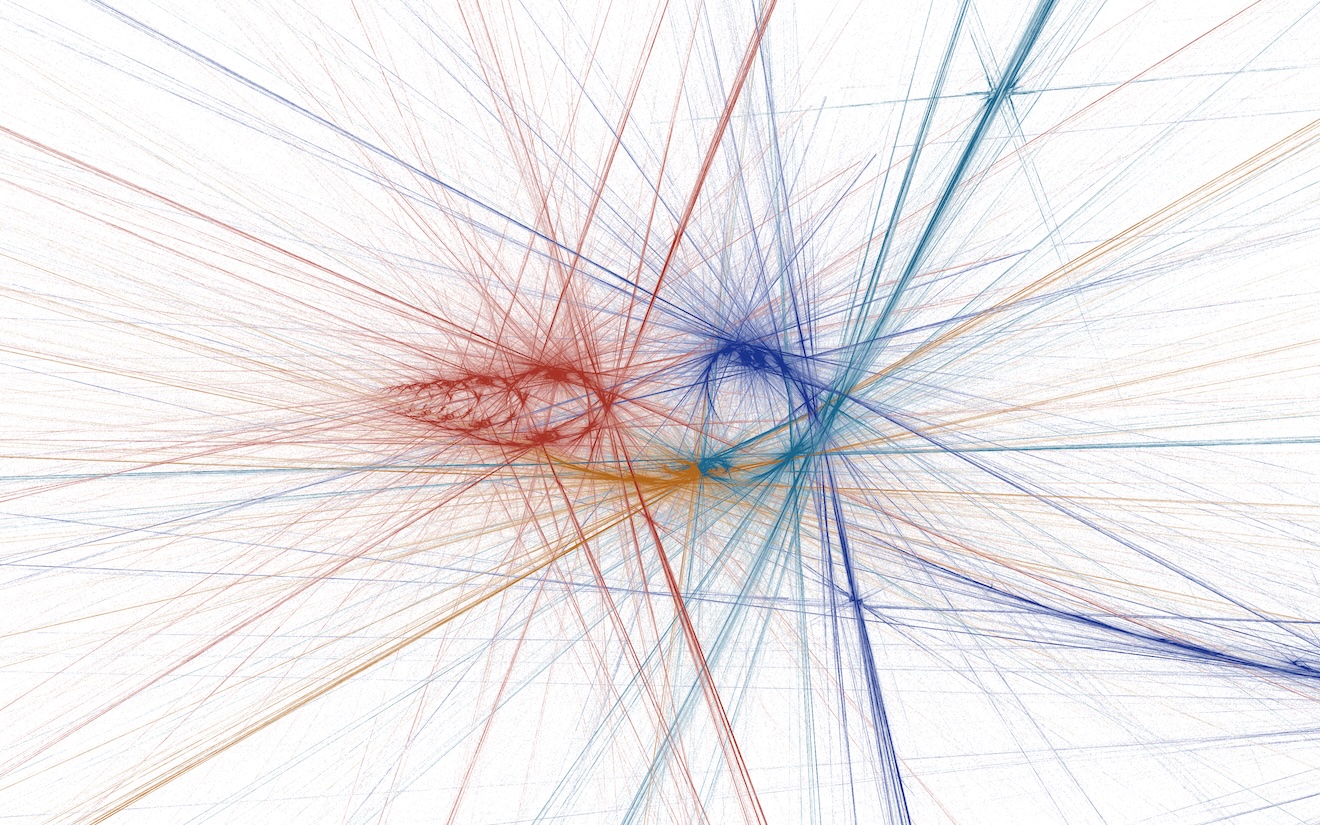}
  \caption{Partial limit sets for A-1 (left, thin)
    and A-15 (right, arithmetic) containing $\approx 20$ billion points.}
  \label{fig:calibration}
\end{figure}

\begin{figure}[t]
  \centering
  \includegraphics[width=0.49\linewidth]{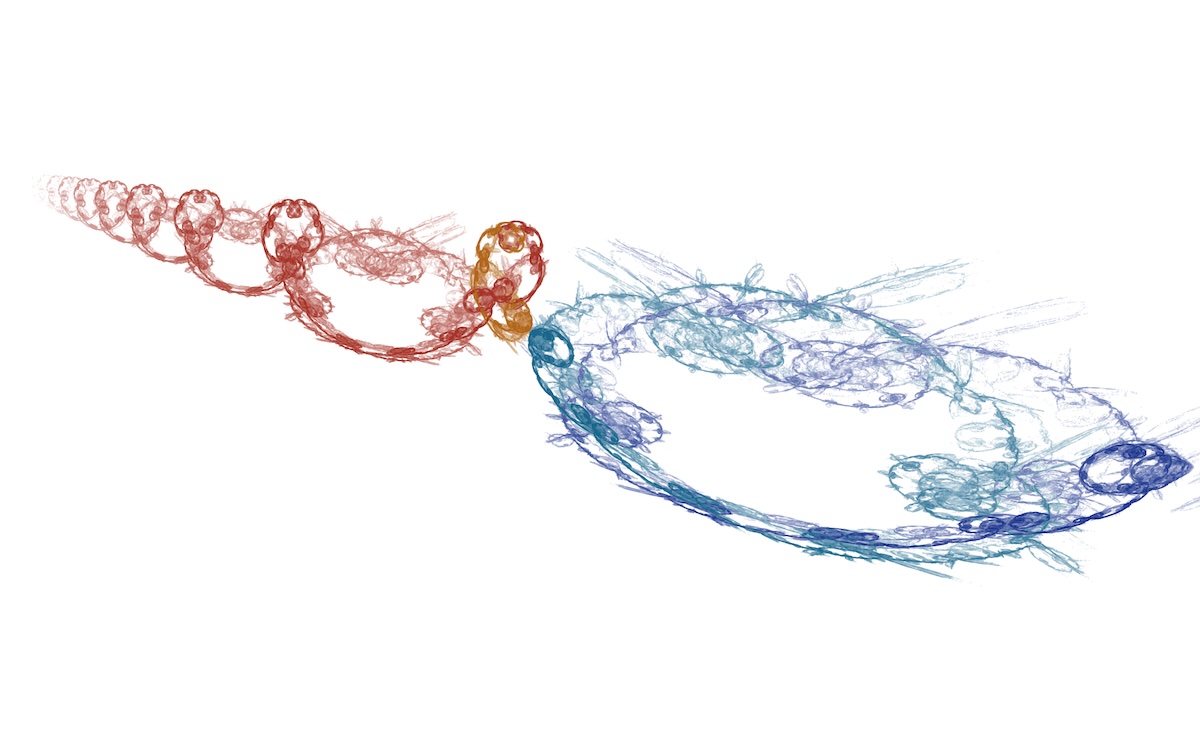}\hfill
  \includegraphics[width=0.49\linewidth]{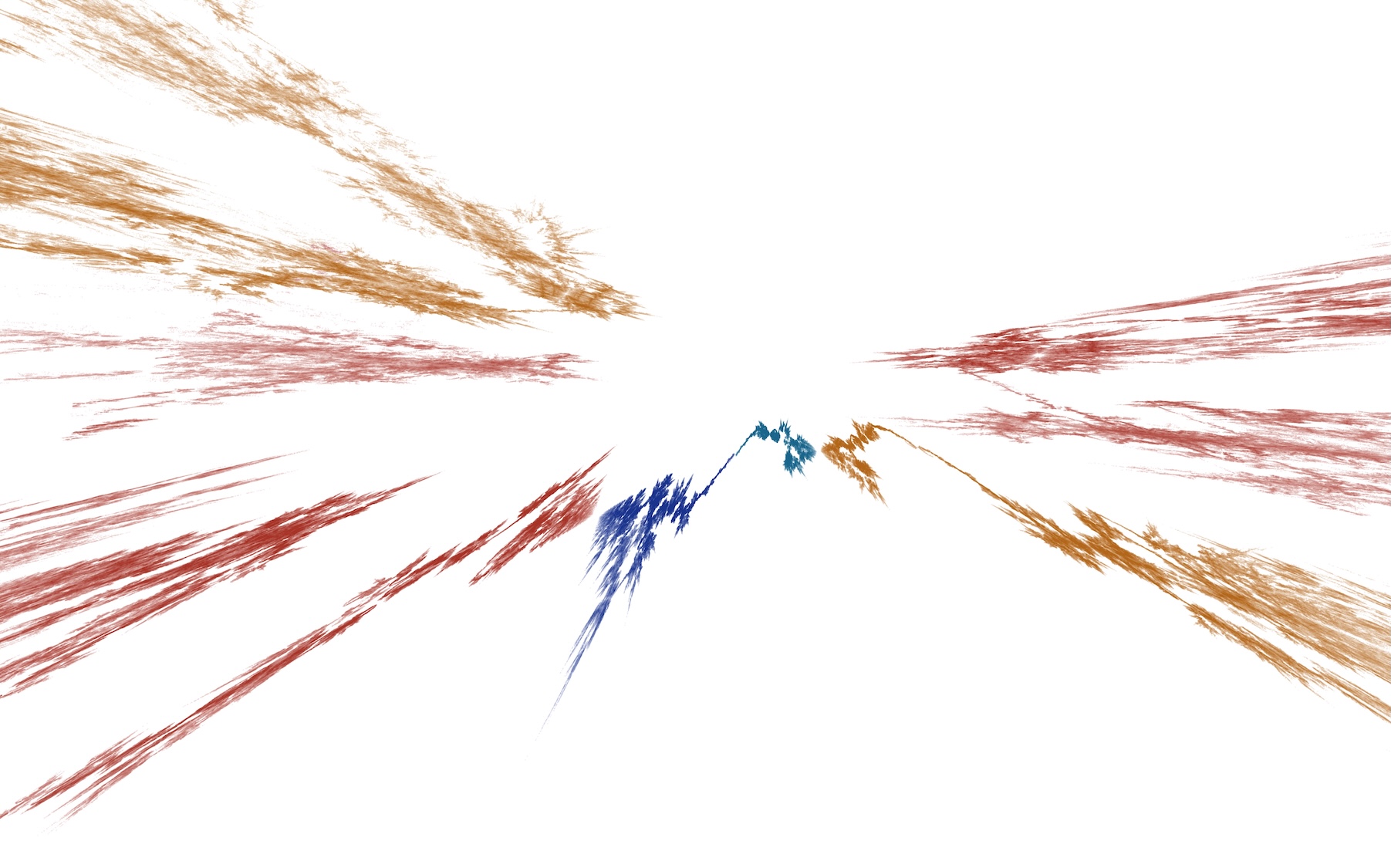}
  \caption{Partial limit sets for C-47 (left) and C-55
    (right) containing  $\approx 20$ billion points.}
  \label{fig:arithmetic}
\end{figure}

\begin{figure}[t]
  \centering
  \includegraphics[width=0.49\linewidth]{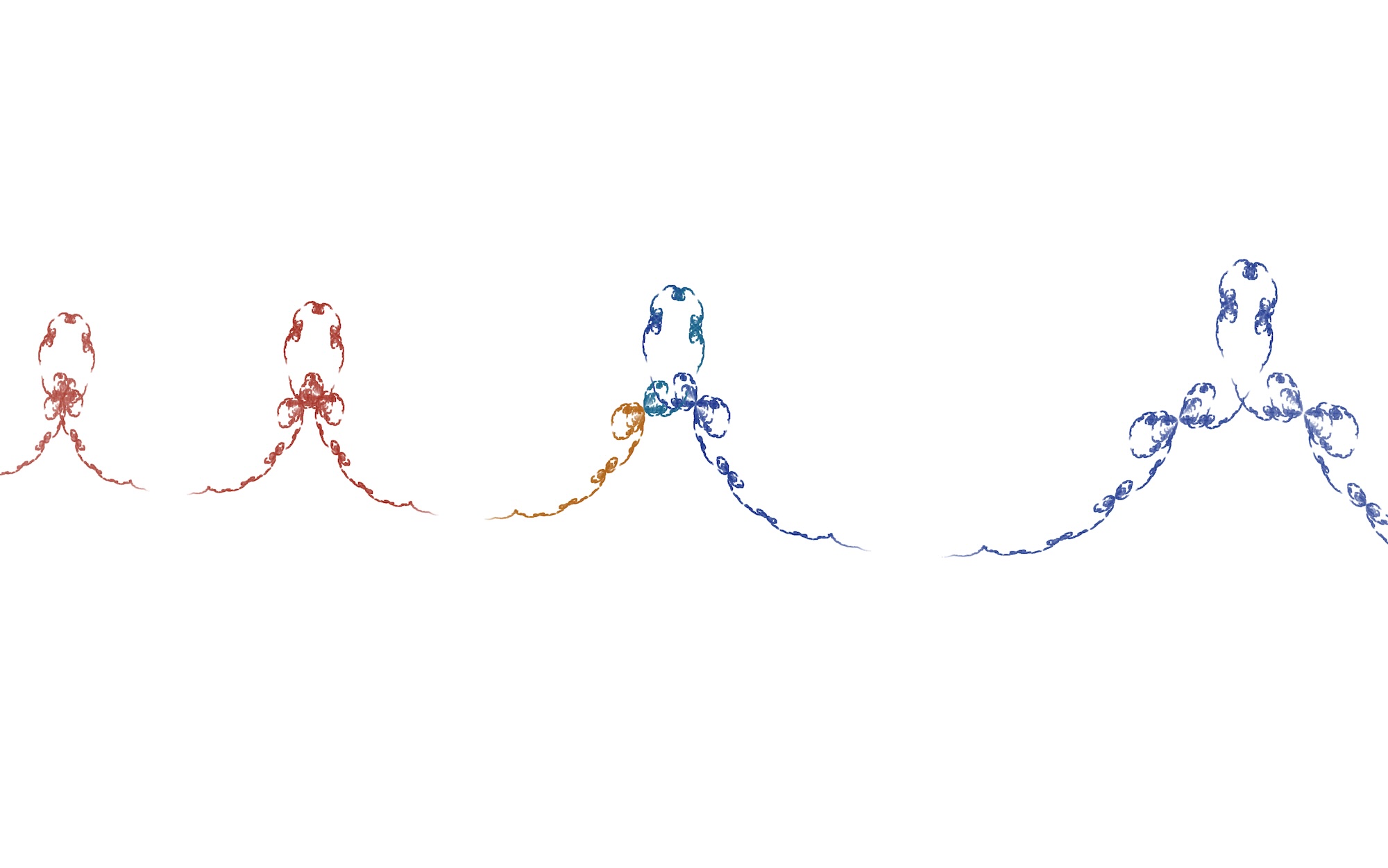}\hfill
   \includegraphics[width=0.49\linewidth]{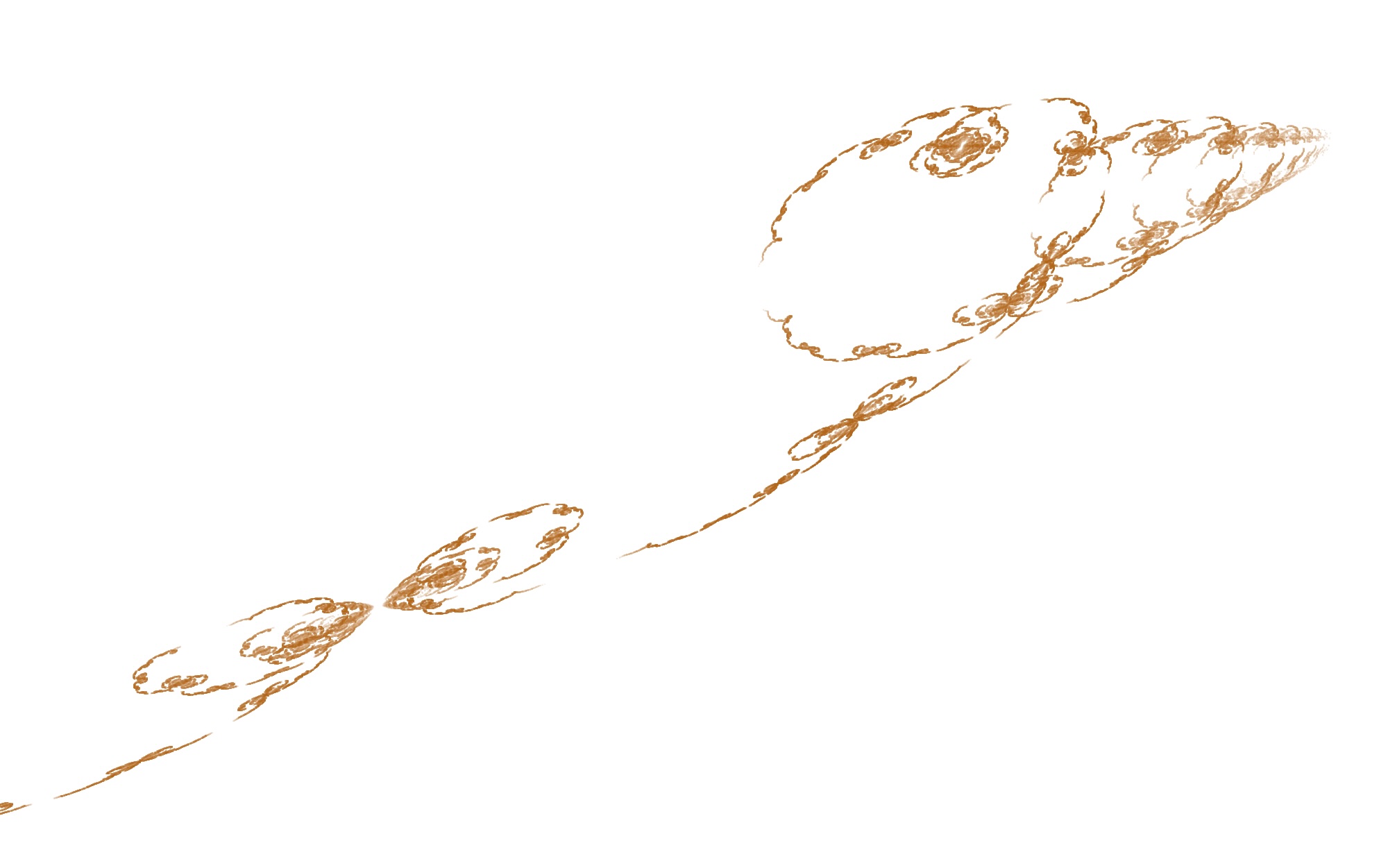}
  \caption{Partial limit set for C-32, containing 20 billion points. The right is a deep deep zoom.}
  \label{fig:c32}
\end{figure}

\begin{figure}[t]
  \centering
  \includegraphics[width=0.32\linewidth]{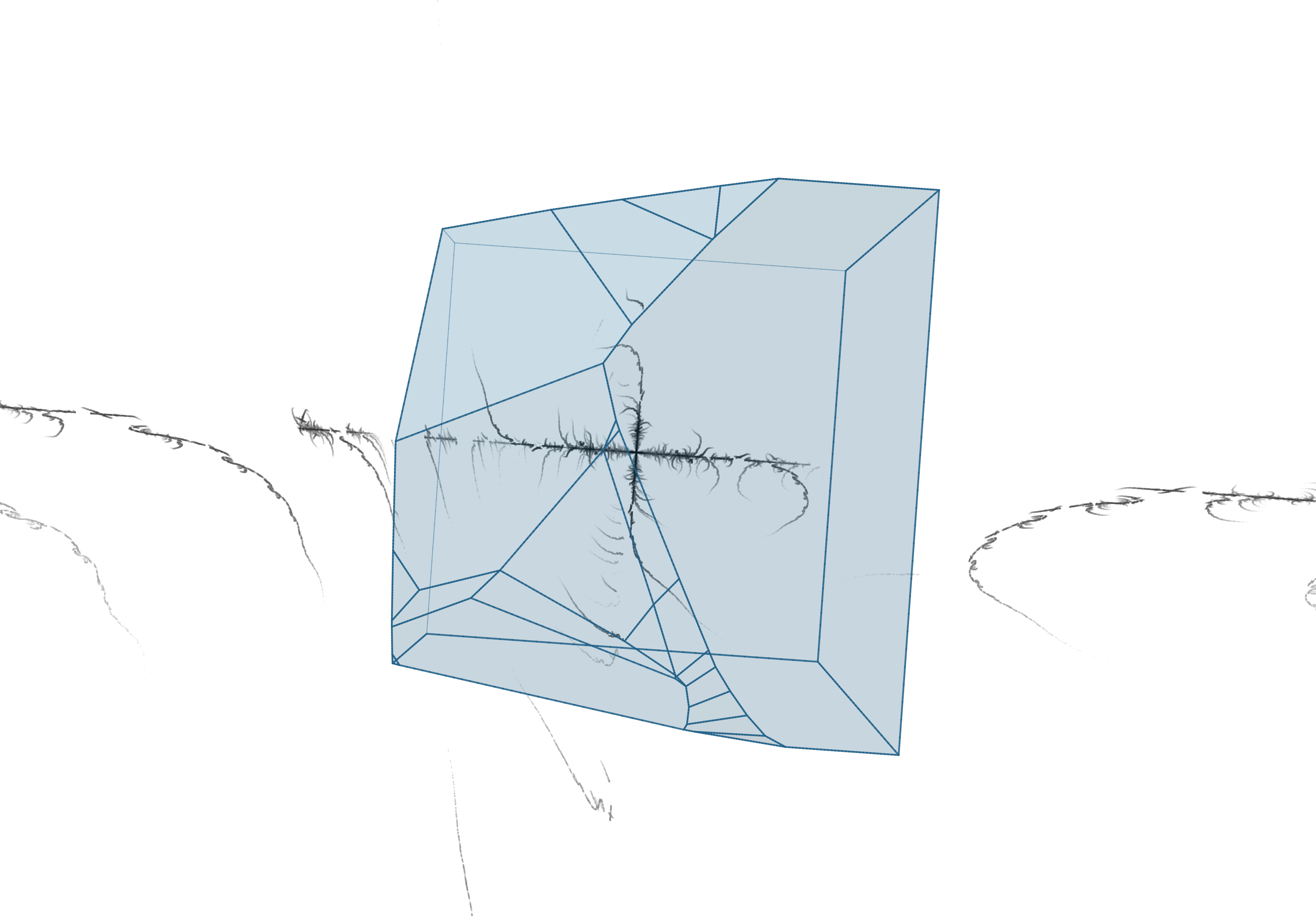}\hfill
  \includegraphics[width=0.32\linewidth]{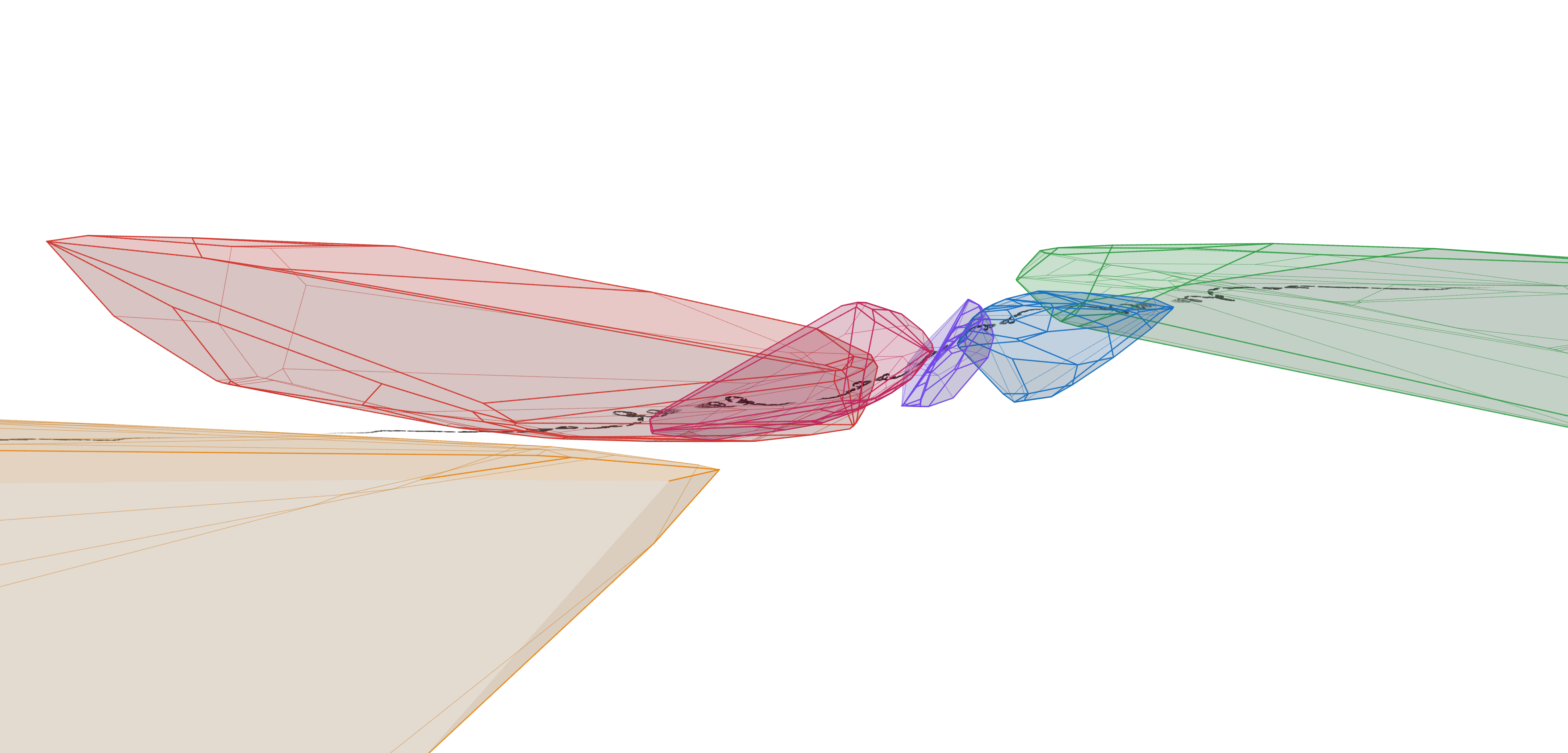}\hfill
  \includegraphics[width=0.32\linewidth]{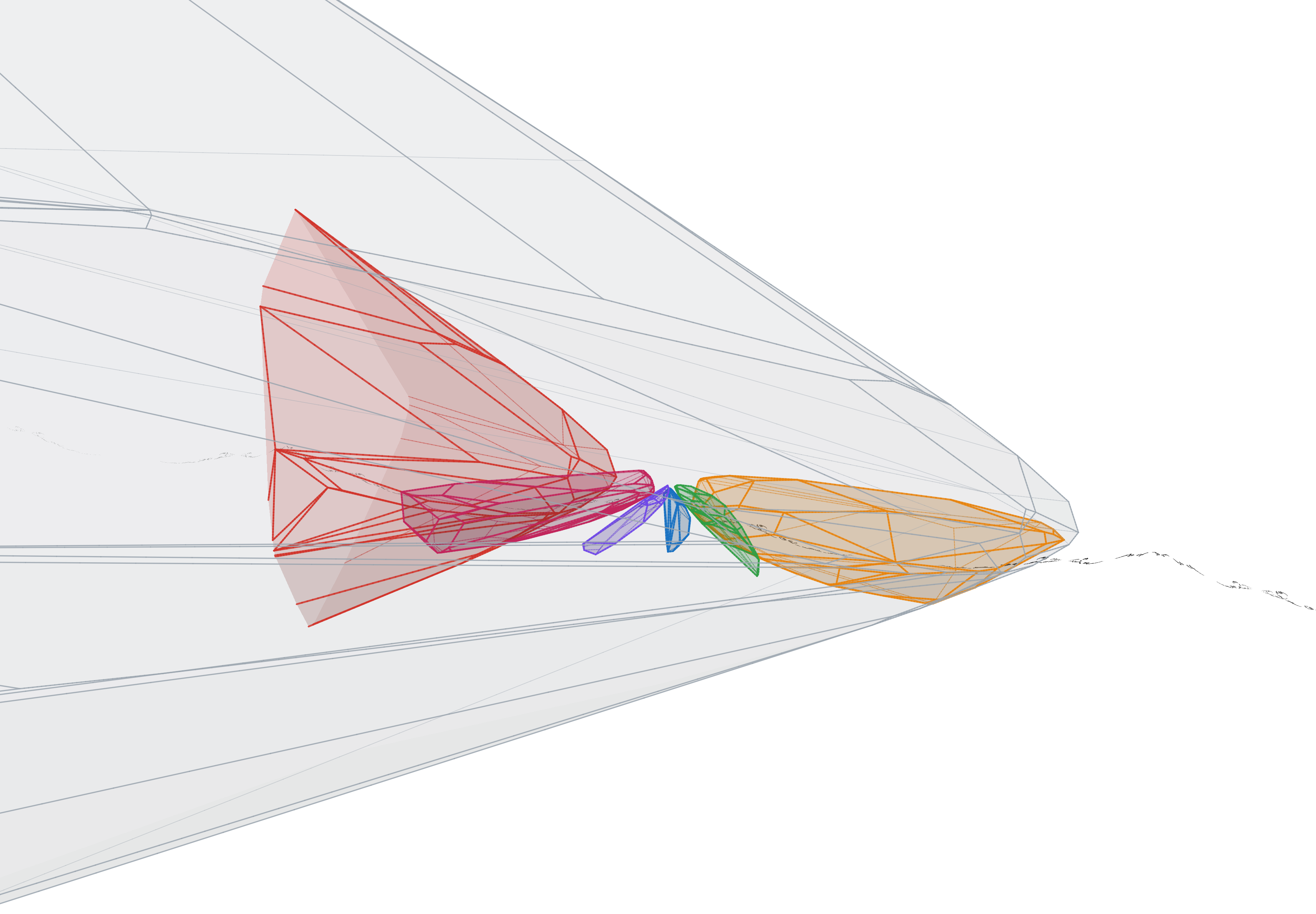}
  \caption{
  Ping-pong domains for C-32. Left: the projectivized
    cone $\PP(K)$ containing a portion of the limit set. Middle: it together with its five nontrivial translates
    $S^k\PP(K)$, $1\le k\le5$. Right: their images
    $T^{-1}S^k\PP(K)$ nested inside the original domain (drawn in a new affine patch),
    illustrating the containments established in
    Lemma~\ref{lem:cone containment}.}
  \label{fig:domains}
\end{figure}

\newpage

\section*{Acknowledgements}

The witness words in this note were found with AlphaEvolve.  DT thanks the
AlphaEvolve team at Google DeepMind for access to the system and for
discussions about auditable AI-assisted mathematical discovery. AW thanks the Hector Fellow Academy for support.

\newpage

\appendix

\section{The C-32 cone certificate}
\label{app:c32-cone}
The cone used in Section~\ref{sec:c-32-thin} is $K=\{y\in\RR^6:Hy\ge0\}$, where

\begin{equation*}
H=\left(\begin{array}{rrrrrr}
1 & -1 & 0 & 0 & 0 & 0 \\
1 & 0 & -1 & 0 & 0 & 0 \\
1 & 0 & 0 & -1 & 0 & 0 \\
1 & 0 & 0 & 0 & -1 & 0 \\
1 & 0 & 0 & 0 & 0 & -1 \\
1 & 0 & 0 & 0 & 0 & 1 \\
1 & 0 & 0 & 0 & 1 & 0 \\
1 & 0 & 0 & 1 & 0 & 0 \\
1 & 0 & 1 & 0 & 0 & 0 \\
1 & 1 & 0 & 0 & 0 & 0 \\
1 & 4 & 7 & 7 & 4 & 1 \\
13 & 11 & 5 & 1 & -5 & -11 \\
0 & -5 & -11 & -14 & -11 & -5 \\
10 & 14 & 11 & 5 & 1 & -5 \\
10 & 5 & -1 & -5 & -11 & -14 \\
4 & -1 & -5 & -11 & -14 & -11 \\
4 & 1 & -5 & -11 & -14 & -11 \\
4 & 11 & 14 & 11 & 5 & -1 \\
4 & 11 & 14 & 11 & 5 & 1 \\
7 & 26 & 43 & 41 & 22 & 5 \\
9 & -4 & -19 & -39 & -45 & -30 \\
9 & 30 & 45 & 39 & 19 & 4 \\
9 & -4 & -21 & -39 & -45 & -30 \\
9 & 30 & 45 & 39 & 21 & 4 \\
15 & -9 & -41 & -78 & -87 & -54 \\
15 & 54 & 87 & 78 & 41 & 9 \\
15 & -9 & -41 & -80 & -87 & -54 \\
15 & 54 & 87 & 80 & 41 & 9 \\
21 & 78 & 130 & 124 & 66 & 15 \\
21 & -15 & -66 & -124 & -132 & -78 \\
21 & 78 & 132 & 124 & 66 & 15 \\
78 & 21 & -15 & -66 & -124 & -132 \\
27 & 101 & 170 & 165 & 90 & 21
\end{array}\right).
\end{equation*}

\newpage

\bibliographystyle{amsalpha}
\bibliography{biblio}

\end{document}